\documentclass[12pt,a4paper]{article}
\usepackage{latexsym,amssymb,amsfonts,amsmath,amsthm,nccmath,float,enumitem,setspace,graphicx}
\usepackage[hidelinks]{hyperref}
\usepackage[margin=2cm]{geometry}
\usepackage[usenames,dvipsnames]{xcolor}

\usepackage[pdf]{pstricks}

\newrgbcolor{VertexColour}{1 1 0}
\newrgbcolor{VertexColourA}{1 0.5 1}
\newrgbcolor{VertexColourB}{0.3 0.4 0.5}
\newrgbcolor{VertexColourC}{1 0.5 0}
\newrgbcolor{VertexColourD}{0.25 0.55 0.75}
\newrgbcolor{VertexColourE}{0.90 0.60 0.00}
\newrgbcolor{EdgeColour}{0 0 0}
\newrgbcolor{EdgeColourOld}{0.5 0.2 0.7}

\newrgbcolor{l2khaki}{0.88 0.86 0.63}
\newrgbcolor{lightkhaki}{0.95 0.93 0.7}

\newtheorem{theorem}{Theorem}[section]

\newtheorem{lemma}[theorem]{Lemma}

\newtheorem{construction}{Construction}
\newtheorem{proposition}[theorem]{Proposition}
\newtheorem{question}{Open Problem}

\theoremstyle{definition}

\def \leq {\leqslant}
\def \geq {\geqslant}
\def \Z {\mathbb{Z}}
\def \L {\mathcal{L}}
\def \circ {{\rm Circ}}

\def \a {{\alpha}}
\def \b {{\beta}}
\def \g {{\gamma}}

\def \mod#1{{\:({\rm mod}\ #1)}}

\let\oldproofname=\proofname
\renewcommand{\proofname}{\rm\bf{\oldproofname}}

\parskip = \baselineskip
\setlist{topsep=-0.7\baselineskip,itemsep=1mm,parsep=0mm,partopsep=0mm,after=\vspace{0.7\baselineskip}}

\title{Decomposable twofold triple systems with non-Hamiltonian 2-block intersection graphs}

\author{Rosalind A. Cameron and David A. Pike \\ Department of Mathematics and Statistics\\ Memorial University of Newfoundland\\ St John's, NL, Canada A1C 5S7\\
\texttt{rahoyte@outlook.com}, \texttt{dapike@mun.ca}}

\date{}

\begin{document}
\sloppy
\def\baselinestretch{1.1}\small\normalsize
\maketitle

\begin{abstract}
The \emph{$2$-block intersection graph} ($2$-BIG) of a twofold triple system (TTS)  is the graph whose vertex set is composed of the blocks of the TTS and two vertices are joined by an edge if the corresponding blocks intersect in exactly two elements. 
The $2$-BIGs are themselves interesting graphs: each component is cubic and $3$-connected, and a $2$-BIG is bipartite exactly when the TTS is decomposable to two Steiner triple systems. Any connected bipartite $2$-BIG with no Hamilton cycle is a counter-example to a conjecture posed by Tutte in 1971.
Our main result is that there exists an integer $N$ such that for all $v\geq N$, if  $v\equiv 1$ or $3\mod{6}$ then there exists a TTS($v$) whose $2$-BIG is bipartite and connected but not Hamiltonian. Furthermore, $13<N\leq 663$. 
Our approach is to construct a TTS($u$) whose $2$-BIG is connected bipartite and non-Hamiltonian and embed it within a TTS($v$)  where $v>2u$ in such a way that, after a single trade, the $2$-BIG of the resulting TTS($v$) is bipartite connected and non-Hamiltonian.
\end{abstract}

\section{Introduction}

A combinatorial design $(V,\mathcal{B})$ consists of a set $V$ of elements (called points), together with a set $\mathcal{B}$ of subsets (called blocks) of $V$. A balanced incomplete block design, ($v, k, \lambda$)-BIBD, is a combinatorial design in which $|V| = v$, for each block $B\in\mathcal{B}$, $|B| = k$, and each $2$-subset of $V$ occurs in precisely $\lambda$ blocks of $\mathcal{B}$. A ($v, 3, 1$)-BIBD is a Steiner triple system of order $v$ (STS($v$)) and a ($v, 3, 2$)-BIBD is a twofold triple system of order $v$ (TTS($v$)).

The block intersection graph  of a design $(V,\mathcal{B})$ is the graph whose vertices are the blocks of $\mathcal{B}$ and two blocks $B_1,B_2\in\mathcal{B}$ are adjacent if $|B_1\cap B_2|>0$. For a nonnegative integer $i$, the $i$-block intersection graph ($i$-BIG) of a design  $(V,\mathcal{B})$ is the graph whose vertices are the blocks of $\mathcal{B}$ and two blocks $B_1,B_2\in\mathcal{B}$ are adjacent if $|B_1\cap B_2|=i$. Note that the $0$-BIG is the complement of the block intersection graph. 
 The block-intersection graph of an STS($v$) is equivalent to the $1$-BIG, and these have been well-studied (for example see \cite{AbueidaPike13,DewarStevens,HorakRosa88}). A Hamilton cycle in the $1$-BIG of an STS induces a minimal change ordering on the blocks of the design, as does a Hamilton cycle in the $2$-BIG of a TTS. Some results on higher $\lambda$ can be found in the work of Asplund and Keranen \cite{AspKer18}. 
In this paper we focus on Hamilton cycles in the $2$-BIG of twofold triple systems.

A Hamilton cycle in the $2$-BIG of a TTS is equivalent to a cyclic Gray code \cite{DewarStevens}, which leads to applications in coding theory. 
It is known that for $v\geq 4$ such that $v\equiv 0,\, 1\mod{3}$ and $v\neq 6$, there exists a TTS($v$) whose $2$-BIG is Hamiltonian \cite{DewarStevens,ErzuPike18}. There also exists a TTS($v$) whose $2$-BIG is connected but non-Hamiltonian when $v=6$ or $v\geq 12$ and $v\equiv 0,1\mod{3}$ \cite{ErzuPike17}. While the existence of TTSs in each case has been established, it is not yet known what properties of a TTS are sufficient for the existence of a Hamilton cycle in the $2$-BIG. We note that, in general, the problem of determining whether cubic graphs are Hamiltonian is NP-complete \cite{GareyNPC76}.

We say a positive integer $v$ is \emph{admissible} if $v\equiv 1$ or $3 \mod{6}$. It is straightforward to show that the $2$-BIG of a TTS($v$) is bipartite if and only if the TTS($v$) can be decomposed into two STS($v$) (see Lemma~\ref{Lemma:BipartiteTTS}).  Our main result is as follows.

\begin{theorem}\label{Theorem:ExistenceLargev}
There exists an integer $N$ such that for all admissible $v\geq N$, there is a TTS($v$) whose $2$-BIG is bipartite connected and non-Hamiltonian. Furthermore, $13<N\leq 663$. 
\end{theorem}

The $2$-BIGs that arise out of Theorem~\ref{Theorem:ExistenceLargev} are all $3$-connected bipartite cubic graphs \cite{ColbournJohnstone}. 
In 1971, Tutte \cite{Tutte71} conjectured that every $3$-connected bipartite cubic graph is Hamiltonian. This was disproved by Horton in the 1970s (see \cite{BondyMurty}). Horton's counter-example and other subsequent counter-examples \cite{EllinghamHorton1983,Georges1989, Horton82} are not $2$-BIGs since they cannot be labelled with the blocks of a TTS. Another conjecture in this strain is Barnette's conjecture which poses the still-open question that every planar $3$-connected bipartite cubic graph is Hamiltonian (see \cite{Grunbaum70}). 
 While the $2$-BIGs we construct are not planar graphs, we note that the labelling induces a cycle double-cover of the graph. 
We also observe that previously studied non-Hamiltonian $2$-BIGs are not bipartite \cite{ErzuPike17} and, conversely, for small examples of TTS($v$) whose $2$-BIG is bipartite and connected the $2$-BIG is also Hamiltonian. The following lemma was obtained by exhaustive computer search.

\begin{lemma}\label{Lemma:smallv} For admissible $v\leq 13$, if the $2$-BIG corresponding to a TTS($v$) is bipartite and connected then it contains a Hamilton cycle. 
\end{lemma}

Our approach to proving Theorem~\ref{Theorem:ExistenceLargev} is to construct a TTS($u$) whose $2$-BIG is connected bipartite and non-Hamiltonian and embed it within a TTS($v$)  where $v>2u$ in such a way that, after a single trade, the $2$-BIG of the resulting TTS($v$) is bipartite connected and non-Hamiltonian. We therefore prove the following Doyen-Wilson type result which is similar to Lindner's result for embedding pairs of Steiner triple systems \cite{Lindner80}.

\begin{theorem}\label{Theorem:ExtendingnonHamiltonianTTS}
Suppose $u$ and $v$ are admissible integers such that $v>2u$ and $u>13$. If there exists a TTS($u$) whose $2$-BIG is bipartite connected and non-Hamiltonian, then there exists a TTS($v$) whose $2$-BIG is bipartite connected and non-Hamiltonian.
\end{theorem}

In order to apply Theorem~\ref{Theorem:ExtendingnonHamiltonianTTS}, in Section~\ref{Section:TTS331} we will construct a TTS(331) whose 2-BIG is bipartite connected and non-Hamiltonian.

\section{Construction of first example of order 331}\label{Section:TTS331}

 We begin by considering the bipartite subgraph illustrated in Figure~\ref{Fig-ConfigurationT}, for which it is a simple exercise to confirm that for any Hamilton cycle of any cubic graph containing this subgraph, the two edges $e_1$ and $e_3$ are both in the cycle or both absent from the cycle
(and similarly for the edges $e_2$ and $e_4$).
We want to associate the vertices of this subgraph with blocks of a partial TTS so that the subgraph is a 2-BIG;
one way of doing so is with the blocks listed in Table~\ref{Tbl-ConfigurationT}.
Any set of sixteen blocks on nine points
having the subgraph of Figure~\ref{Fig-ConfigurationT} as its 2-BIG
will be denoted as configuration $\mathbb{T}$.

\begin{figure}[htbp]
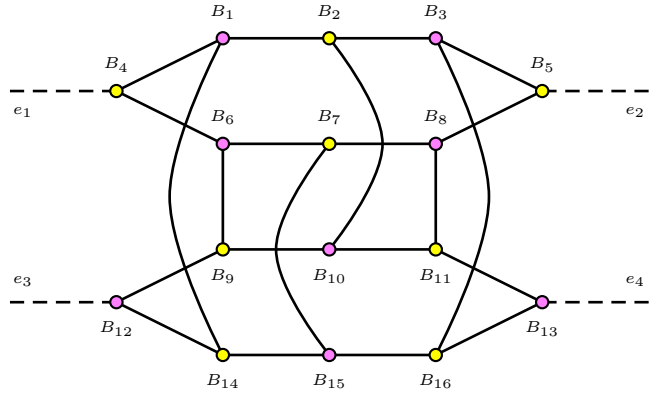

\begin{center}

\pspicture(0,1)(8.4,6)           
\psset{unit=0.7cm}

\psline[linecolor=EdgeColour,linewidth=1pt](2,3)(4,2)
\psline[linecolor=EdgeColour,linewidth=1pt](2,3)(4,4)
\psline[linecolor=EdgeColour,linewidth=1pt](4,2)(8,2)
\psline[linecolor=EdgeColour,linewidth=1pt](4,4)(8,4)
\psline[linecolor=EdgeColour,linewidth=1pt](10,3)(8,2)
\psline[linecolor=EdgeColour,linewidth=1pt](10,3)(8,4)

\psline[linecolor=EdgeColour,linewidth=1pt](4,4)(4,6)
\psline[linecolor=EdgeColour,linewidth=1pt](8,4)(8,6)

\pscurve[linecolor=EdgeColour,linewidth=1pt](6,2)(5,4)(6,6)
\pscurve[linecolor=EdgeColour,linewidth=1pt](6,4)(7,6)(6,8)

\pscurve[linecolor=EdgeColour,linewidth=1pt](4,2)(3,5)(4,8)
\pscurve[linecolor=EdgeColour,linewidth=1pt](8,2)(9,5)(8,8)

\psline[linecolor=EdgeColour,linewidth=1pt](2,7)(4,6)
\psline[linecolor=EdgeColour,linewidth=1pt](2,7)(4,8)
\psline[linecolor=EdgeColour,linewidth=1pt](4,6)(8,6)
\psline[linecolor=EdgeColour,linewidth=1pt](4,8)(8,8)
\psline[linecolor=EdgeColour,linewidth=1pt](10,7)(8,6)
\psline[linecolor=EdgeColour,linewidth=1pt](10,7)(8,8)

\psline[linecolor=EdgeColour,linewidth=1pt,linestyle=dashed](2,3)(0,3)
\psline[linecolor=EdgeColour,linewidth=1pt,linestyle=dashed](2,7)(0,7)
\psline[linecolor=EdgeColour,linewidth=1pt,linestyle=dashed](10,3)(12,3)
\psline[linecolor=EdgeColour,linewidth=1pt,linestyle=dashed](10,7)(12,7)

\pscircle[fillstyle=solid,fillcolor=VertexColourA](2,3){0.95mm}
\pscircle[fillstyle=solid,fillcolor=VertexColour](2,7){0.95mm}
\pscircle[fillstyle=solid,fillcolor=VertexColourA](10,3){0.95mm}
\pscircle[fillstyle=solid,fillcolor=VertexColour](10,7){0.95mm}

\pscircle[fillstyle=solid,fillcolor=VertexColour](4,2){0.95mm}
\pscircle[fillstyle=solid,fillcolor=VertexColourA](6,2){0.95mm}
\pscircle[fillstyle=solid,fillcolor=VertexColour](8,2){0.95mm}

\pscircle[fillstyle=solid,fillcolor=VertexColour](4,4){0.95mm}
\pscircle[fillstyle=solid,fillcolor=VertexColourA](6,4){0.95mm}
\pscircle[fillstyle=solid,fillcolor=VertexColour](8,4){0.95mm}

\pscircle[fillstyle=solid,fillcolor=VertexColourA](4,6){0.95mm}
\pscircle[fillstyle=solid,fillcolor=VertexColour](6,6){0.95mm}
\pscircle[fillstyle=solid,fillcolor=VertexColourA](8,6){0.95mm}

\pscircle[fillstyle=solid,fillcolor=VertexColourA](4,8){0.95mm}
\pscircle[fillstyle=solid,fillcolor=VertexColour](6,8){0.95mm}
\pscircle[fillstyle=solid,fillcolor=VertexColourA](8,8){0.95mm}

\rput[c](4,8.5){\tiny $B_1$}
\rput[c](6,8.5){\tiny $B_2$}
\rput[c](8,8.5){\tiny $B_3$}

\rput[c](2,7.5){\tiny $B_4$}
\rput[c](10,7.5){\tiny $B_5$}

\rput[c](4,6.5){\tiny $B_6$}
\rput[c](6,6.5){\tiny $B_7$}
\rput[c](8,6.5){\tiny $B_8$}

\rput[c](4,3.5){\tiny $B_9$}
\rput[c](6,3.5){\tiny $B_{10}$}
\rput[c](8,3.5){\tiny $B_{11}$}

\rput[c](2,2.5){\tiny $B_{12}$}
\rput[c](10,2.5){\tiny $B_{13}$}

\rput[c](4,1.5){\tiny $B_{14}$}
\rput[c](6,1.5){\tiny $B_{15}$}
\rput[c](8,1.5){\tiny $B_{16}$}

\rput[c](0.25,6.6){\tiny $e_1$}
\rput[c](11.75,6.6){\tiny $e_2$}

\rput[c](0.25,3.4){\tiny $e_3$}
\rput[c](11.75,3.4){\tiny $e_4$}

\endpspicture
\end{center}
\caption{The 2-BIG of configuration $\mathbb{T}$.}
\label{Fig-ConfigurationT}
\end{figure}

\begin{table}[htbp]
\begin{center}
$
\begin{array}{cc@{\hspace*{15mm}}cc}
B_1 &   \{4,5,6\} & B_9 &   \{3,5,9\}\\
B_2 &   \{4,5,8\} & B_{10} &   \{5,8,9\}\\
B_3 &   \{4,7,8\} & B_{11} &   \{1,8,9\}\\
B_4 &   \{3,4,6\} & B_{12} &   \{2,3,5\}\\
B_5 &   \{1,4,7\} & B_{13} &   \{1,2,8\}\\
B_6 &   \{3,6,9\} & B_{14} &   \{2,5,6\}\\
B_7 &   \{6,7,9\} & B_{15} &   \{2,6,7\}\\
B_8 &   \{1,7,9\} & B_{16} &   \{2,7,8\}

\end{array}
$
\end{center}
\caption{Blocks for configuration $\mathbb{T}$.}
\label{Tbl-ConfigurationT}
\end{table}

We initially discovered configuration $\mathbb{T}$ by observing that it occurs within some TTSs of order 13.
To construct a TTS(331) with the properties we desire,
we will perform a number of operations that entail embeddings of configuration $\mathbb{T}$.
As an intermediate goal, we exploit the twinned behaviour of edges $e_1$ and $e_3$ to develop a configuration that forbids Hamilton cycles.
We then show how to embed such an obstructing configuration into a TTS.

To some extent our approach is similar to the constructions of other counter-examples to Tutte's conjecture,
whereby a bipartite cubic Hamiltonian graph with a twinned pair of edges
is used to construct a bipartite cubic connected non-Hamiltonian graph
(for instance, see~\cite{EllinghamHorton1983,Georges1989}
as well as
Exercise 4.2.14 of~\cite{BondyMurty}).
However, we additionally require our graphs to be 2-BIGs of (partial) TTSs,
which is not the case with these known counter-examples to Tutte's conjecture.

In Figure~\ref{Fig-ConfigurationX} we depict a configuration on 34 blocks and 16 points, denoted as configuration $\mathbb{X}$.
For the $\mathbb{T}$ configuration on the left (for which only the four vertices with external neighbours are shown),
use the 16 blocks from Table~\ref{Tbl-ConfigurationT}.
For the blocks of the $\mathbb{T}$ configuration on the right, apply the function
$f_\mathbb{X}$ to each block of Table~\ref{Tbl-ConfigurationT}, where $f_\mathbb{X}$
maps point 1 (resp.~2, 3, 4) to 11 (resp.~1, 2, 10) and $f_\mathbb{X}(x)=x+7$ for each point $x \in \{5,\ldots,9\}$.
The other two blocks are  $\{3,4,10\}$ and $\{2,3,10\}$.
Observe that for any cubic graph having the 2-BIG of configuration $\mathbb{X}$ as a subgraph, the edge $e_\mathbb{X}$ must be in every Hamilton cycle.
Moreover, vertex $\{2,3,10\}$ is incident with $e_\mathbb{X}$.

\begin{figure}[htbp]
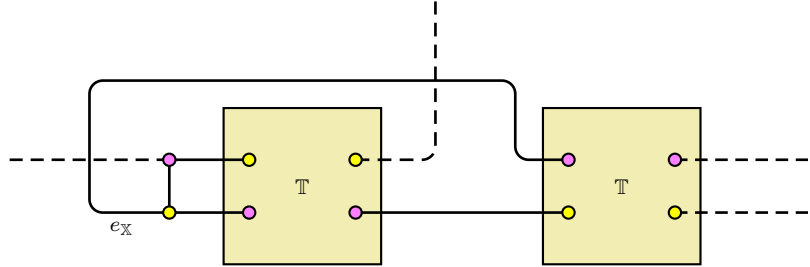

\begin{center}
\pspicture(1,2)(7,6)
\psset{unit=0.7cm}

\psframe[fillstyle=solid,fillcolor=lightkhaki](2,3)(5,6)
\psframe[fillstyle=solid,fillcolor=lightkhaki](8,3)(11,6)
\rput[c](3.5,4.5){\tiny\scriptsize $\mathbb{T}$}
\rput[c](9.5,4.5){\tiny\scriptsize $\mathbb{T}$}

\psline[linecolor=EdgeColour,linewidth=1pt](1,4)(1,5)
\psline[linecolor=EdgeColour,linewidth=1pt](1,5)(2.5,5)
\psline[linecolor=EdgeColour,linewidth=1pt](1,4)(2.5,4)

\psline[linecolor=EdgeColour,linewidth=1pt](4.5,4)(8.5,4)
\psline[linecolor=EdgeColour,linewidth=1pt,linearc=0.25,showpoints=false](1,4)(-0.5,4)(-0.5,6.5)(7.5,6.5)(7.5,5)(8.5,5)

\psline[linecolor=EdgeColour,linewidth=1pt,linestyle=dashed](10.5,5)(13,5)
\psline[linecolor=EdgeColour,linewidth=1pt,linestyle=dashed](10.5,4)(13,4)
\psline[linecolor=EdgeColour,linewidth=1pt,linestyle=dashed](-2,5)(1,5)
\psline[linecolor=EdgeColour,linewidth=1pt,linestyle=dashed,linearc=0.25](4.5,5)(6,5)(6,8)

\pscircle[fillstyle=solid,fillcolor=VertexColourA](1,5){0.95mm}
\pscircle[fillstyle=solid,fillcolor=VertexColour](1,4){0.95mm}

\pscircle[fillstyle=solid,fillcolor=VertexColourA](2.5,4){0.95mm}
\pscircle[fillstyle=solid,fillcolor=VertexColour](2.5,5){0.95mm}
\pscircle[fillstyle=solid,fillcolor=VertexColour](4.5,5){0.95mm}
\pscircle[fillstyle=solid,fillcolor=VertexColourA](4.5,4){0.95mm}

\pscircle[fillstyle=solid,fillcolor=VertexColour](8.5,4){0.95mm}
\pscircle[fillstyle=solid,fillcolor=VertexColourA](8.5,5){0.95mm}
\pscircle[fillstyle=solid,fillcolor=VertexColourA](10.5,5){0.95mm}
\pscircle[fillstyle=solid,fillcolor=VertexColour](10.5,4){0.95mm}

\rput[c](0.1,3.7){\tiny\scriptsize $e_\mathbb{X}$}

\endpspicture
\end{center}
\caption{Configuration $\mathbb{X}$}
\label{Fig-ConfigurationX}
\end{figure}

In Figure~\ref{Fig-ConfigurationP} we illustrate configuration $\mathbb{P}$, which has 36 blocks and 16 points.
For the $\mathbb{T}$ configuration on the left, we again use the 16 blocks from Table~\ref{Tbl-ConfigurationT}.
For the blocks of the $\mathbb{T}$ configuration on the right, apply the function
$f_\mathbb{P}$ to each block of Table~\ref{Tbl-ConfigurationT}, where $f_\mathbb{P}$
maps point 1 (resp.~2, 3, 4) to 2 (resp.~1, 11, 10) and $f_\mathbb{P}(x)=x+7$ for each point $x \in \{5,\ldots,9\}$.
The other four blocks are  $\{3,4,10\}$, $\{2,3,10\}$, $\{4,10,11\}$ and $\{1,4,11\}$.
Note that 2-BIG of configuration $\mathbb{P}$ is a bipartite cubic graph in which the edge $e_\mathbb{P}$ is in every Hamilton cycle.
Moreover, vertex $\{2,3,10\}$ is incident with $e_\mathbb{P}$.

\begin{figure}[htbp]
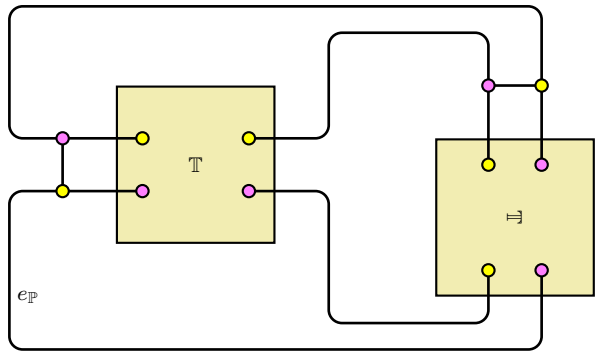

\begin{center}
\pspicture(1,1)(6,6)
\psset{unit=0.7cm}

\psframe[fillstyle=solid,fillcolor=lightkhaki](2,3)(5,6)
\psframe[fillstyle=solid,fillcolor=lightkhaki](8,2)(11,5)
\rput[c](3.5,4.5){\tiny\scriptsize $\mathbb{T}$}
\rput[c](9.5,3.5){\tiny\scriptsize \rotateright{$\mathbb{T}$}}

\psline[linecolor=EdgeColour,linewidth=1pt](1,4)(1,5)
\psline[linecolor=EdgeColour,linewidth=1pt](1,5)(2.5,5)
\psline[linecolor=EdgeColour,linewidth=1pt](1,4)(2.5,4)

\psline[linecolor=EdgeColour,linewidth=1pt](9,6)(10,6)
\psline[linecolor=EdgeColour,linewidth=1pt](9,6)(9,4.5)
\psline[linecolor=EdgeColour,linewidth=1pt](10,6)(10,4.5)

\psline[linecolor=EdgeColour,linewidth=1pt,linearc=0.25](4.5,4)(6,4)(6,1.5)(9,1.5)(9,2.5)
\psline[linecolor=EdgeColour,linewidth=1pt,linearc=0.25](4.5,5)(6,5)(6,7)(9,7)(9,6)
\psline[linecolor=EdgeColour,linewidth=1pt,linearc=0.25](1,5)(0,5)(0,7.5)(10,7.5)(10,6)
\psline[linecolor=EdgeColour,linewidth=1pt,linearc=0.25](1,4)(0,4)(0,1)(10,1)(10,2.5)

\pscircle[fillstyle=solid,fillcolor=VertexColourA](1,5){0.95mm}
\pscircle[fillstyle=solid,fillcolor=VertexColour](1,4){0.95mm}
\pscircle[fillstyle=solid,fillcolor=VertexColourA](9,6){0.95mm}
\pscircle[fillstyle=solid,fillcolor=VertexColour](10,6){0.95mm}

\pscircle[fillstyle=solid,fillcolor=VertexColourA](2.5,4){0.95mm}
\pscircle[fillstyle=solid,fillcolor=VertexColour](2.5,5){0.95mm}
\pscircle[fillstyle=solid,fillcolor=VertexColour](4.5,5){0.95mm}
\pscircle[fillstyle=solid,fillcolor=VertexColourA](4.5,4){0.95mm}

\pscircle[fillstyle=solid,fillcolor=VertexColour](9,4.5){0.95mm}
\pscircle[fillstyle=solid,fillcolor=VertexColourA](10,4.5){0.95mm}
\pscircle[fillstyle=solid,fillcolor=VertexColour](9,2.5){0.95mm}
\pscircle[fillstyle=solid,fillcolor=VertexColourA](10,2.5){0.95mm}

\rput[c](0.35,2){\tiny\scriptsize $e_\mathbb{P}$}

\endpspicture
\end{center}
\caption{Configuration $\mathbb{P}$}
\label{Fig-ConfigurationP}
\end{figure}

We will next employ several constructions that bear similarities to one that is attributed to D.A.~Holton in~\cite{Ellingham1981}.
Holton's construction begins with a bipartite cubic graph $G_1$
having a vertex $v_1$ with neighbours $x_1$, $y_1$ and $z_1$ such that the edge $\{v_1, z_1\}$ is in no Hamilton cycle.
Let $G_2$ also be a bipartite cubic graph for which there is a vertex $v_2$ with neighbours $x_2$, $y_2$ and $z_2$ such that the edge
$\{v_2,z_2\}$ is in no Hamilton cycle.
For each $i \in \{1,2\}$ remove vertex $v_i$ from $G_i$, leaving a severed edge
dangling from each of $x_i$, $y_i$ and $z_i$, and then bind these dangling edges together to create a bipartite cubic graph
on $|V(G_1)| + |V(G_2)| -2$ vertices.  Holton is careful to bind the severed edge dangling from $x_1$ to the one dangling from $z_2$,
and to bind the severed edge dangling from $z_1$ to the one dangling from $x_2$, thereby ensuring that the resultant graph is
non-Hamiltonian.

We will call operations of this nature {\it splicing} operations;
such terminology was previously used in~\cite{Pike1997}.
Although the description of Holton's construction as presented in~\cite{Ellingham1981} requires that $G_2$ be an isomorphic copy of $G_1$,
we will allow $G_1$ and $G_2$ to be non-isomorphic.
We will also relax the cubic requirement, allowing one of the graphs $G_1$ and $G_2$ to be a non-regular graph of maximum degree 3,
as would be the case for the 2-BIG of configuration $\mathbb{X}$.
Moreover, although Holton's construction specifies an edge in each of $G_1$ and $G_2$ that is in no Hamilton cycle,
we will begin with edges that have the property of being in every Hamilton cycle, as is the case with the
edges $e_\mathbb{X}$ and $e_\mathbb{P}$ of configurations $\mathbb{X}$ and $\mathbb{P}$ respectively.

For our first splicing operation, take two copies of configuration $\mathbb{P}$, say $\mathbb{P}_1$ and $\mathbb{P}_2$.
For each $i \in \{1,2\}$, remove a vertex $v_i$ that is incident with an edge $e_i$ that must be in every Hamilton cycle of the 2-BIG of $\mathbb{P}_i$.
Now bind the six dangling edges so that $e_1$ and $e_2$ are not bound to each other.
In the resulting graph, which we denote $\mathbb{P} \asymp \mathbb{P}$, 
the edges arising from each of $e_1$ and $e_2$ have the property of being in every Hamilton cycle in $\mathbb{P} \asymp \mathbb{P}$.
Hence the third edge that is formed during the splicing operation has the property of being in no Hamilton cycle.

What remains to be confirmed is that $\mathbb{P} \asymp \mathbb{P}$ can be realised as the 2-BIG of a configuration of 70 blocks of a partial TTS.
For configuration $\mathbb{P}_1$ take the 36 blocks on point set $\{1,2,\ldots,16\}$ that are used for our initial description of the $\mathbb{P}$ configuration.
For the blocks of configuration $\mathbb{P}_2$, apply the function $f_{\mathbb{P}_2}$ to each block of $\mathbb{P}_1$,
where $f_{\mathbb{P}_2}$ maps point 2 (resp.~3, 10) to 3 (resp.~2, 10)
and the remaining points of $\mathbb{P}_1$ ({\it viz.},  $1, 4, 5, \ldots, 9, 11, 12, \ldots, 16$) are mapped to $17, 18, 19, \ldots, 29$, respectively.
Each of $\mathbb{P}_1$ and $\mathbb{P}_2$ contains a block $\{2,3,10\}$, which is the block that we delete from each of $\mathbb{P}_1$ and $\mathbb{P}_2$
as we commence the splicing operation.  The binding of severed edges is now naturally determined by shared pairs of points, so that
vertices $\{3,4,10\}$, $\{2,3,5\}$ and $\{2,10,14\}$ of $\mathbb{P}_1$ are made adjacent to
$\{3,10,27\}$, $\{2,3,19\}$ and $\{2,10,18\}$ of $\mathbb{P}_2$, respectively.
Hence $\mathbb{P} \asymp \mathbb{P}$ does indeed correspond to a configuration of blocks as desired.
Moreover, the edge between $\{2,3,5\}$ and $\{2,3,19\}$
is in no Hamilton cycle of the 2-BIG of $\mathbb{P} \asymp \mathbb{P}$.

We now perform a second splicing operation, this time working with a copy of configuration $\mathbb{X}$ and a copy of configuration $\mathbb{P}$.
For the $\mathbb{X}$ configuration, 
take the 34 blocks on point set $\{1,2,\ldots,16\}$ that are used for our initial description of the $\mathbb{X}$ configuration.
For the blocks of the $\mathbb{P}$ configuration, apply the same function $f_{\mathbb{P}_2}$ that was used in building $\mathbb{P} \asymp \mathbb{P}$
to each block of our initial $\mathbb{P}$ configuration.
Again, delete two instances of block $\{2,3,10\}$ and then let $\mathbb{X} \asymp \mathbb{P}$ denote the resulting configuration 
on 68 blocks and 29 points.  
It now follows that the edge between $\{2,3,5\}$ and $\{2,3,19\}$
is in no Hamilton cycle of any cubic graph that has the 2-BIG of $\mathbb{X} \asymp \mathbb{P}$ as a subgraph.

For our third and final splicing operation we will use $\mathbb{X} \asymp \mathbb{P}$ together with $\mathbb{P} \asymp \mathbb{P}$,
each of which contains an edge that is in no Hamilton cycle (as is the case with Holton's construction).
For the $\mathbb{X} \asymp \mathbb{P}$ configuration,
take the $\mathbb{X} \asymp \mathbb{P}$ configuration on point set $\{1,2,\ldots,29\}$ that is described in the preceding paragraph.
For the blocks of the $\mathbb{P} \asymp \mathbb{P}$ configuration, apply the function $f_{\mathbb{P} \asymp \mathbb{P}}$ to each block 
of the $\mathbb{P} \asymp \mathbb{P}$ configuration that was previously described on point set $\{1,2,\ldots,29\}$,
where $f_{\mathbb{P} \asymp \mathbb{P}}$ maps point 2 (resp.~3, 5) to 5 (resp.~3, 2)
and the remaining points ({\it viz.},  $1, 4, 6, 7, \ldots, 29$) are mapped to $30, 31, 32, \ldots, 55$, respectively.
Now remove the two instances of the block $\{2,3,5\}$
and denote the resulting configuration of 136 blocks on 55 points by $\mathbb{F}$.

Observe that any cubic graph that has the 2-BIG of configuration $\mathbb{F}$ as a subgraph cannot be Hamiltonian.
So our goal now is to embed configuration $\mathbb{F}$ into a TTS, thereby ensuring that the 2-BIG of the TTS is non-Hamiltonian. We also require that the 2-BIG of the TTS is bipartite and  we show in the following lemma that this is equivalent to the TTS being composed from two STSs.

\begin{lemma}\label{Lemma:BipartiteTTS}
The $2$-BIG of a (partial) TTS($v$), $(V,\mathcal{B})$ is bipartite if and only if the blocks of $\mathcal{B}$ can be partitioned into the blocks of two (partial) STS($v$), $(V,\mathcal{S}_1)$ and $(V,\mathcal{S}_2)$.
\end{lemma}

\begin{proof} First suppose that the blocks of $\mathcal{B}$ can be partitioned into two (partial) STS($v$). If there is an edge $\{X,Y\}$ in the $2$-BIG of $\mathcal{B}$ then the blocks corresponding to $X$ and $Y$ have two points in common and must be in distinct partial STS($v$). Hence the $2$-BIG is bipartite. 

Now suppose that the $2$-BIG corresponding to $(V,\mathcal{B})$  is bipartite with partition set $\mathcal{S}_1$ and $\mathcal{S}_2$. Consider the set of blocks given by vertices in $\mathcal{S}_1$. No pair of points occurs more than once in this set of blocks, otherwise there would be an edge between the corresponding vertices in $\mathcal{S}_1$. Thus, because $\mathcal{B}$ is a (partial) TTS($v$), the blocks corresponding to vertices in $\mathcal{S}_1$ form a (partial) STS($v$). Likewise, the blocks corresponding to vertices of $\mathcal{S}_2$ form a (partial) STS($v$).
\end{proof}

We note that  the 2-BIG of $\mathbb{F}$ is a bipartite graph and $\mathbb{F}$  contains no repeated blocks. By Lemma~\ref{Lemma:BipartiteTTS}, the blocks in $\mathbb{F}$ can therefore be partitioned into the blocks of two block-disjoint partial STS. Here we turn to a result by Lindner that shows that $\mathbb{F}$ can indeed be embedded into a simple TTS of order 331 (and many larger orders as well).

\begin{theorem}[See Theorem 7.2 of~\cite{Lindner80}]\label{Theorem:Lindner}
Let $(U,\mathcal{P}_1)$ and $(U,\mathcal{P}_2)$ be partial STS($u$). 
Then for every admissible $v \geq 6u+1$, there exists a pair of STS($v$), $(V,\mathcal{S}_1)$ and $(V,\mathcal{S}_2)$ 
such that 
$(U,\mathcal{P}_1)$ is embedded in $(V,\mathcal{S}_1)$,
$(U,\mathcal{P}_2)$ is embedded in $(V,\mathcal{S}_2)$
and $\mathcal{P}_1 \cap \mathcal{P}_2 = \mathcal{S}_1 \cap \mathcal{S}_2$.
\end{theorem}

However, this result by Lindner does not provide any assurance that a connected 2-BIG will result from any of the
TTSs that are produced, whereas we specifically seek 2-BIGs that are connected.
We therefore implemented the construction outlined in the proof of Lindner's theorem
(which itself incorporates constructions by 
Cruse~\cite{Cruse1974}, Evans~\cite{Evans1960} and Ryser~\cite{Ryser1951} when $v \equiv 1$ (mod 6))
and built an actual instance of a TTS(331) in which configuration $\mathbb{F}$ is embedded.
It was then a relatively straightforward task to confirm that the 2-BIG of this particular TTS(331) is indeed connected.
As some steps in the construction entail elements of choice, we provide 
the full list of 36410 blocks of our TTS(331) within a supplementary data file.



\section{Embedding TTS and maintaining connected 2-BIG}
In this section we prove our main result, Theorem~\ref{Theorem:ExistenceLargev}.
Our overall approach is to obtain Theorem~\ref{Theorem:ExtendingnonHamiltonianTTS} which is a generalised Doyen-Wilson result for constructing a TTS with connected bipartite non-Hamiltonian 2-BIGs. We then apply this result to the TTS($331$) constructed in Section~\ref{Section:TTS331}.

We will first use collections of difference triples to construct a pair of block-disjoint partial STS($v$). The $2$-BIG of the resulting partial TTS($v$) is clearly bipartite since the partial TTS($v$) is decomposable to two partial STS($v$)s, and we use properties of the difference triples to show that it is connected.  Furthermore, the resulting partial TTS($v$) can be completed to a TTS($v$) by adding the blocks of a TTS($u$).  We show that non-Hamiltonicity of the $2$-BIG is preserved in this construction by first constructing a TTS($v$) whose $2$-BIG has two components, one of which corresponds to the original TTS($u$), and then performing a single trade to connect the two components.


The following notation and preliminary results will be required throughout the construction.

A \emph{circulant} graph with vertex set $\Z_w$ is a graph with edge set $E=\{\{i,i+d\mod{w}\}: i\in \Z_w, d\in S\}$ where $S\subseteq \{1,\ldots, \lfloor \frac{w}{2}\rfloor \}$ is the generating set of differences; we denote this graph by $\circ(w,S)$. 
In a circulant graph $\circ(w,S)$, the \emph{order} of an edge $\{i,i+d\}$ is the order of $d$ in $\Z_w$. We say that an edge in $\circ(w, S)$ is \emph{even} if it has even order.

For a $1$-factor $F$ on vertex set $W$ and a vertex $\infty\not\in W$, we define the following blocks on $W\cup\{\infty\}$ as $F\vee\infty=\{\{x,y,\infty\}: \{x,y\} \in F\}$. 
We will use the following results to construct $1$-factorisations from which we build blocks of the form $F\vee \infty$, where $F$ is a $1$-factor of $\mathbb{Z}_w$.

\begin{lemma}[{\cite[Lemma 2]{SternLenz80}}]\label{Lemma:1factorisationSL}
Let $G$ be a circulant graph with vertex set $\mathbb{Z}_w$. If $G$ contains an edge of even order, then $G$ can be $1$-factorised.
\end{lemma}

\begin{lemma}\label{Lemma:1factorisationCP} Let $w$ be a positive even integer and let $G$ be a circulant graph $\circ(w,S)$ where $S$ contains distinct elements $c$ and $d$ such that $\gcd(w,c)=1$ and $d$ has even order in $\mathbb{Z}_w$.
Then $G$ has a $1$-factorisation $\mathcal{F}$, and  there are two $1$-factors $F_1$ and $F_2$ in $\mathcal{F}$ such that $F_1\cup F_2$ is a Hamilton cycle in $G$.
\end{lemma}

\begin{proof}
$\circ(w, S)$ contains edges $\{i,i+c\}$ for $i\in \Z_w$. Let $F_1=\left\{\left\{i,i+c\right\}: i\in\left\{0,2,\ldots,w-2\right\}\right\}$ and $F_2=\left\{\left\{i,i+c\right\}: i\in \left\{1,3,\ldots,w-1\right\}\right\}$. Then $\circ(w,S\setminus\{c\})$ has a $1$-factorisation by Lemma~\ref{Lemma:1factorisationSL}.
\end{proof}

A \emph{difference triple} is an ordered $3$-tuple $(a,b,c)$ that satisfies $a+b=c$. Two difference triples are disjoint if they have no elements in common.
For an even integer $w$ and a set $\mathcal{D}$ of pairwise disjoint difference triples on a subset of the differences $\{1,2,\ldots,\tfrac{w}{2}-1\}$, the partial STS($w$) induced by $\mathcal{D}$ is given by $(\Z_w,\mathcal{S})$, where $\mathcal{S}=\left\{\left\{i,b+i,c+i\right\}:i\in\Z_w, (a,b,c)\in \mathcal{D}\right\}$.

\subsection{The TTS($v$) construction}

Our embedding will rely on the following sets of difference triples. Constructions~\ref{Construction:sOddSL} and \ref{Construction:sEvenSL} are given by Stern and Lenz in their proof of the Doyen-Wilson Theorem \cite{SternLenz80}.

\begin{construction}\label{Construction:sOddSL} Let $w=12t+k$ where $k\in \{0,2,4\}$ and $t>0$. We construct the $2t-1$ difference triples
\[
\begin{array}{ll}
(1,3t-1,3t),& (2,5t-1,5t+1),\\
(3,3t-2,3t+1),& (4,5t-2,5t+2),\\
\vdots & \vdots\\
(2t-3,2t+1,4t-2),& (2t-2,4t+1,6t-1),\\
(2t-1,2t,4t-1).&
\end{array}
\]
Note that in these difference triples, every difference $d\in\{1,2,\ldots,6t\}\setminus\{4t,5t,6t\}$ occurs exactly once.
\end{construction}

\begin{construction}\label{Construction:sOddCP} Let $w=12t+k$ where $k\in \{0,2,4\}$ and $t>0$. We construct the $2t-1$ difference triples
\[
\begin{array}{ll}
(1,5t,5t+1),& (2,3t-1,3t+1),\\
(3,5t-1,5t+2),& (4,3t-2,3t+2),\\
\vdots & \vdots\\
(2t-3,4t+2,6t-1),& (2t-2,2t+1,4t-1),\\
(2t-1,4t+1,6t).&
\end{array}
\]
Note that in these difference triples, every difference $d\in\{1,2,\ldots,6t\}\setminus\{2t,3t,4t\}$ occurs exactly once.
\end{construction}

\begin{construction}\label{Construction:sEvenSL} Let $w=12t+k$ where $k\in \{6,8,10\}$ and $t>0$.  We construct the $2t$ difference triples
\[
\begin{array}{ll}
(1,5t+2,5t+3),& (2,3t,3t+2),\\
(3,5t+1,5t+4),& (4,3t-1,3t+3),\\
\vdots & \vdots\\
(2t-1,4t+3,6t+2),& (2t,2t+1,4t+1).\\
\end{array}
\]
Note that in these difference triples, every difference $d\in\{1,2,\ldots,6t+3\}\setminus\{3t+1,4t+2,6t+3\}$ occurs exactly once.
\end{construction}

\begin{construction}\label{Construction:sEvenCP}
Let $w=12t+k$ where $k\in \{6,8,10\}$ and $t>0$. We construct the $2t$ difference triples
\[
\begin{array}{ll}
(1,3t+1,3t+2),& (2,5t+2,5t+4),\\
(3,3t,3t+3),& (4,5t+1,5t+5),\\
\vdots & \vdots\\
(2t-1,2t+2,4t+1),& (2t,4t+3,6t+3).\\
\end{array}
\]
Note that in these difference triples, every difference $d\in\{1,2,\ldots,6t+3\}\setminus\{2t+1,4t+2,5t+3\}$ occurs exactly once.
\end{construction}

We now give the construction that will be used to prove Theorem~\ref{Theorem:ExtendingnonHamiltonianTTS}.

\begin{construction}\label{Construction:PartialTTS}
\rm
Suppose $u$ and $v$ are admissible integers such that $v>2u$ and $u>13$. Let $U=\{\infty_1,\ldots,\infty_u\}$ and $V=U\cup \mathbb{Z}_{w}$, where $w=v-u$. We construct a pair of partial STS($v$), $(V,\mathcal{R}_1)$ and $(V,\mathcal{R}_2)$ so that $\mathcal{R}_1\cup \mathcal{R}_2$ can be completed to the block set of a TTS($v$) by adding the blocks of a TTS($u$) with point set $U$. 
Throughout the following we will assume addition is modulo $w$.

Let $k$, $t$ and $s$ be integers such that $w=12t+k$ where $k\in\{0,2,\ldots,10\}$, $s=2t$ if $k\geq 6$ and $s=2t-1$ if $k\leq 4$. Let $m=s+1$.

If $w\equiv 2\mod{6}$ then let $u=7+6h$ where $h\in\{2,\ldots,s-2\}$, and if $w\equiv 4\mod{6}$ then $u=9+6h$ where $h\in\{1,\ldots,s-2\}$.
If $w\equiv 0\mod{6}$ then let $u=1+6h$ where $h\in\{3,4,\ldots,s-2\}$,
 or $u=3+6h$ where $h\in\{2,3,\ldots,s-2\}$.

\begin{itemize}
\item[{\rm Step 1.}] If $w\equiv 0\mod{6}$ then we add the following blocks to the partial STS($v$)s. 
\begin{itemize}
\item[$\bullet$] If $u\equiv 1\mod{6}$, then we include  in $\mathcal{R}_1$ the $4m$ blocks $\{i,m+i,2m+i\}$, $\{2m+i,3m+i,4m+i\}$, $\{4m+i,5m+i,i\}$ and $\{m+i,3m+i,5m+i\}$ where $i\in\{0,\ldots,s\}$. Similarly, we include in $\mathcal{R}_2$ the $4m$ blocks $\{i,m+i,5m+i\}$, $\{m+i,2m+i,3m+i\}$, $\{3m+i,4m+i,5m+i\}$ and $\{i,2m+i,4m+i\}$ where $i\in\{0,\ldots,s\}$. Note that these blocks only use pairs of difference $m$ and $2m$.

\item[$\bullet$] If $u\equiv 3\mod{6}$, then for $i\in\{0,\ldots,m-1\}$, add a copy of the blocks $\{i,2m+i,4m+i\}$ and $\{m+i,3m+i,5m+1\}$ to both $\mathcal{R}_1$ and $\mathcal{R}_2$. Note that these blocks only use pairs of difference $2m$. We also note that in Step 5 these blocks will be traded away from $\mathcal{R}_1$ to ensure that the triple system has no repeated blocks.
\end{itemize}

\item[{\rm Step 2.}] We define two sets of difference triples $\mathcal{D}_1$ and $\mathcal{D}_2$, and add the blocks of the partial STS($v$) induced by these difference triples to $\mathcal{R}_1$ and $\mathcal{R}_2$ respectively. 

Note that, when relevant, we will not use the differences required in Step 1. 
For $i\in\{1,2\}$, we define $S_i$ to be the set of differences in $\{1,2,\ldots,\tfrac{w}{2}\}$ that are not in a difference triple of $\mathcal{D}_i$. Let $S'=S_1\cap S_2$, let $S_1^*=S_1\setminus S_2$ and $S_2^*=S_2\setminus S_1$. 

The sets $\mathcal{D}_1$ and $\mathcal{D}_2$ are defined as follows.

\begin{itemize}
\item[$\bullet$] If $w-u\in\{1,3,5\}$ then let $\mathcal{D}_1=\mathcal{D}_2=\emptyset$. Clearly $\{1,\tfrac{w}{2}\}\subseteq S'$ and $S_1^*=S_2^*=\emptyset$.
\item[$\bullet$] If $w-u\in\{7,9,11\}$ and $20\leq w\leq 26$ then let $\mathcal{D}_1=\{(2,3,5)\}$ and $\mathcal{D}_2=\{(2,7,9)\}$. Note that $\{1,\tfrac{w}{2}\}\subseteq S'$, $S_1^*=\{7,9\}$ and $S_2^*=\{3,5\}$.
\item[$\bullet$] If $w-u\in\{7,9,11\}$ and $w\geq 28$ then let $\mathcal{D}_1=\{(2,2t+3,2t+5)\}$ and $\mathcal{D}_2=\{(2,2t+7,2t+9)\}$. Note that $\{1,\tfrac{w}{2}\}\subseteq S'$, $S_1^*=\{2t+7,2t+9\}$ and $S_2^*=\{2t+3,2t+5\}$.
\end{itemize}
Now suppose $w-u>11$, the $\mathcal{D}_1$ and $\mathcal{D}_2$ are defined as follows.

\begin{itemize}
\item[$\bullet$] If $k\in\{6,8,10\}$ then let $\mathcal{D}'_1$ and $\mathcal{D}'_2$ be the collections of difference triples given by Constructions~\ref{Construction:sEvenSL} and \ref{Construction:sEvenCP} respectively.

If $k=6$ let $\mathcal{D}_1$ and $\mathcal{D}_2$ be the sets obtained by removing difference triples containing the differences $s$ and $1,2,\ldots,h-1$ from $\mathcal{D}'_1$ and $\mathcal{D}'_2$ respectively. Then note that $\{1,\tfrac{w}{2},d\}\subseteq S'$, for some $d$ of even order in $\mathbb{Z}_w$. If $u\equiv 3\mod{6}$ then $d=m$, if $u\equiv 1\mod{6}$ and $t$ even then $d=5t+3$, and if $u\equiv 1\mod{6}$ and $t$ odd then $d=5t+2$ (recall that $h\geq 3$ when $u\equiv 1\mod{6}$). Also note that $2m-1\in S^*_1$ and $2m+1\in S_2^*$. 

If $k\in\{8,10\}$ let $\mathcal{D}_1$ and $\mathcal{D}_2$ be the sets obtained by removing difference triples containing the differences $1,2,\ldots,h$ from $\mathcal{D}'_1$ and $\mathcal{D}'_2$ respectively. Note that $\{1,\tfrac{w}{2}\}\subseteq S'$, $6t+3\in S_1^*$ and $2t+1\in S_2^*$. 

\item[$\bullet$]  If $k\in\{0,2,4\}$ then let $\mathcal{D}'_1$ and $\mathcal{D}'_2$ be the collection of difference triples given by Constructions~\ref{Construction:sOddSL} and \ref{Construction:sOddCP} respectively. 

If $k=0$ or $t$ is odd and $k=2$, we remove the difference triples containing the differences $s$ and $1,2,\ldots,h-1$, the resulting sets are $\mathcal{D}_1$ and $\mathcal{D}_2$. Note that  if $k=0$ then $\{1,\tfrac{w}{2},d\}\subseteq S'$, for some $d$ of even order in $\mathbb{Z}_w$ ($d=3t$ if $u\equiv 1\mod{6}$ and $d=m$ if $u\equiv 3\mod{6}$), and if $k=2$ then $\{2t-1,\tfrac{w}{2}\}\subseteq S'$. Also note that $2m-1\in S^*_1$ and $2m+1\in S_2^*$ in each case.

If $t$ is even and $k\in\{2,4\}$, then we obtain $\mathcal{D}_1$ and $\mathcal{D}_2$ by removing difference triples containing the difference $1$ and $s,s-1,\ldots,s-h+1$ from $\mathcal{D}'_1$ and $\mathcal{D}'_2$ . Note that $\{1,\tfrac{w}{2}\}\subseteq S'$, $3t-1\in S_1^*$ and $5t+1\in S_2^*$. 

Finally, if $k=4$ and $t$ is odd then we remove difference triples containing $1,2,\ldots,h$ to obtain $\mathcal{D}_1$ and $\mathcal{D}_2$. Note that $\{1,\tfrac{w}{2}\}\subseteq S'$, $6t\in S_1^*$ and $2t\in S_2^*$. 

\end{itemize}

\item[{\rm Step 3.}] We define the graphs $\mathcal{L}'=\circ(w,S')$, and for $i\in\{1,2\}$, let $\mathcal{L}_i^*=\circ(w,S^*_i)$. Note that, for $i\in\{1,2\}$, the edges in $\mathcal{L}_i=\circ(w,S'\cup S_i^*)$ correspond to pairs in $\mathbb{Z}_w$ that are not already in blocks of $\mathcal{R}_i$. 

We now define $1$-factorisations of the graphs $\mathcal{L}_1$ and $\mathcal{L}_2$. 
We will use Lemmas~\ref{Lemma:1factorisationSL} and \ref{Lemma:1factorisationCP}  to show that there exists a $1$-factorisation $\{F_1,F_2,F_3,F_4\}$ of $\mathcal{L}_1^*$, a $1$-factorisation $\{G_1,G_2,G_3,G_4\}$ of $\mathcal{L}_2^*$, and a $1$-factorisation $\{F_5,\ldots,F_u\}$ of $\L'$.

\begin{itemize}
\item[$\bullet$] When $w\equiv 0\mod{6}$ we have $\{1,\tfrac{w}{2},d\}\subseteq S'$ where $d$ has even order in $\mathbb{Z}_w$, and otherwise we have  $\{g,\tfrac{w}{2}\}\subseteq S'$, where $\gcd(g,w)=1$. Thus, by Lemma~\ref{Lemma:1factorisationCP} there exists a $1$-factorisation $\{F_5,\ldots,F_u\}$ of $\mathcal{L}'$ so that $F_u\cup F_{u-1}$ is a Hamilton cycle on $\mathbb{Z}_w$. Furthermore if  $w\equiv 0\mod{6}$ then we specify the following $1$-factors. 
If $u\equiv 1\mod{6}$, then $F_{u-2}=\left\{\left\{x,x+\tfrac{w}{2}\right\}:x\in\{0,1,\ldots,\tfrac{w}{2}-1\}\right\}$. If $w\equiv 6\mod{12}$ and $u\equiv 3\mod{6}$, then $F_{u-2}=\left\{\left\{x,x+m\right\}:x\in\{0,2,\ldots,w-2\}\right\}$ and $F_{u-3}=\left\{\left\{x,x+m\right\}:x\in\{1,3,\ldots,w-1\}\right\}$. Finally if $w\equiv 0\mod{12}$ and $u\equiv 3\mod{6}$, then $F_{u-2}=\left\{\left\{x,x+m\right\}:x\in X \right\}$ and $F_{u-3}=\left\{\left\{x,x-m\right\}:x\in X\right\}$, where $X=\left\{0,1,\ldots,m-1\right\}\cup \left\{2m,\ldots,3m-1\right\}\cup \left\{4m,\ldots,5m-1\right\}$.
Note that if $w\equiv 0\mod{6}$ then we remove these specified $1$-factors from the graph before applying Lemma~\ref{Lemma:1factorisationCP}.

\item[$\bullet$] If  $w-u\in\{7,9,11\}$, then for $S_i^*=\{c_i,d_i\}$ we define the following $1$-factorisations.
A $1$-factorisation of $\L_1^*$ is given by 
\begin{align*}
F_1&=\{\{i,i+c_1\}:i\in\{0,2,\ldots,w-2\}\},\\
F_2&=\{\{i,i+c_1\}:i\in\{1,3,\ldots,w-1\}\},\\
F_3&=\{\{i,i+d_1\}:i\in\{0,2,\ldots,w-2\}\}, \\
F_4&=\{\{i,i+d_1\}:i\in\{1,3,\ldots,w-1\}\}.
\end{align*}
A $1$-factorisation of $\L_2^*$ is given by 
\begin{align*}
G_1&=\{\{i,i+c_2\}:i\in\{0,2,\ldots,w-2\}\},\\
G_2&=\{\{i,i+c_2\}:i\in\{1,3,\ldots,w-1\}\},\\
G_3&=\{\{i,i+d_2\}:i\in\{1,3,\ldots,w-1\}\},\\
G_4&=\{\{i,i+d_2\}:i\in\{0,2,\ldots,w-2\}\}.
\end{align*}

\item[$\bullet$] If $w-u>11$, then note that for $i\in\{1,2\}$, $S_i^*$ has an element of even order in $\Z_w$ so by Lemma~\ref{Lemma:1factorisationSL} there exists a $1$-factorisation of $\L_i^*$.
\end{itemize}

\item[{\rm Step 4.}] We construct blocks of the form $F\vee \infty$, where $F\in \mathcal{F}$  and $\infty\in U$.

Add the blocks $F_i\vee\infty_i$ for $i\in \{1,2,\ldots,u\}$ to $\mathcal{R}_1$. Also add the following blocks to $\mathcal{R}_2$: $G_i\vee\infty_j$ where $(i,j)\in\{(1,u),(2,1),(3,2),(4,3)\}$ and $F_i\vee\infty_{i-1}$ for $i\in\{5,6,\ldots,u\}$.


\item[{\rm Step 5.}] If $w\equiv 0\mod{6}$ then we perform the following trades on blocks that were added to $\mathcal{R}_1$ in Steps 1--3. 

If $u\equiv 1\mod{6}$ then, for each $i\in \{0,\ldots,m-1\}$, we remove the following blocks from $\mathcal{R}_1$
\[
\begin{array}{ccc}
\{i,m+i,2m+i\},&\{2m+i,3m+i,4m+i\},&\{i,4m+i,5m+i\},\\ 
\{i,3m+i,\infty_{u-2}\},&\{m+i,4m+i,\infty_{u-2}\},&\{2m+i,5m+i,\infty_{u-2}\},
\end{array}\]
and replace them with the blocks
\[
\begin{array}{ccc}
\{i,3m+i,4m+i\},&\{i,2m+i,5m+i\},&\{m+i,2m+i,4m+i\},\\
\{i,m+i,\infty_{u-2}\},&\{2m+i,3m+i,\infty_{u-2}\}, &\{4m+i,5m+i,\infty_{u-2}\}. 
\end{array}\]

If $u\equiv 3\mod{6}$ then, for each $i\in\{0,1,\ldots,m-1\}$, we remove the following blocks from $\mathcal{R}_1$ where $(j,j')=(u-2,u-3)$ if  $m$ even or if $m$ odd and $i$ even, and $(j,j')=(u-3,u-2)$ if $m$ odd and $i$ odd 
\[\begin{array}{ccc}
\{i,2m+i,4m+i\},& \{m+i,3m+i,5m+i\}, &\{i,m+i,\infty_j\}\\ \{2m+i,3m+i,\infty_j\},& \{5m+i,i,\infty_{j'}\},& \{3m+i,4m+i,\infty_{j'}\} 
\end{array}\]

and replace them with the blocks
\[\begin{array}{ccc}
\{i,2m+i,\infty_j\},&\{m+i,3m+i,\infty_j\},& \{i,m+i,5m+i\}, \\
\{2m+i,3m+i,4m+i\}, &\{4m+i,i,\infty_{j'}\},& \{3m+i,5m+i,\infty_{j'}\}. \end{array}\]
\end{itemize}
This concludes Construction~\ref{Construction:PartialTTS}; $(V,\mathcal{R}_1\cup\mathcal{R}_2)$ is the required partial TTS($v$). We will show in Subsection~\ref{Subsection:connected} that its $2$-BIG is connected.
\end{construction}

\subsection{Connected bipartite $2$-BIG}\label{Subsection:connected}

We now need some additional results that we will use to prove that the $2$-BIG of the partial TTS($v$), $(V,\mathcal{R}_1\cup\mathcal{R}_2)$ given by Construction~\ref{Construction:PartialTTS} is connected. 

We first make the following observation regarding the difference triples in Constructions~\ref{Construction:sOddSL}--\ref{Construction:sEvenCP}. 
For a partial TTS($v$) generated by difference triples $D_1, D_2,\ldots, D_\sigma$, we define the \emph{orbit graph} $\mathcal{O}$ of the corresponding $2$-BIG as follows. Take $V(\mathcal{O})=\{D_1,D_2,\ldots,D_\sigma\}$ and let $D_i$ and $D_j$ be adjacent in $\mathcal{O}$ if they have at least one element in common. 	

We are interested in pairs of consecutive difference triples given by Constructions~\ref{Construction:sOddSL} and \ref{Construction:sOddCP} (or Constructions~\ref{Construction:sEvenSL} and \ref{Construction:sEvenCP}), by which we mean four triples of the form $(a,b,a+b),(a,c,a+c),(a+1,c-1,a+c)$ and $(a+1,b,a+b+1)$.

\begin{proposition}\label{Prop:ConnectedOrbitGraph}
Any two or more consecutive pairs of difference triples given by Constructions~\ref{Construction:sOddSL} and \ref{Construction:sOddCP} (or Constructions~\ref{Construction:sEvenSL} and \ref{Construction:sEvenCP}) induces a connected orbit graph.
\end{proposition}
\begin{proof}
Difference triples of the form $(a,b,a+b),(a,c,a+c),(a+1,c-1,a+c)$ and $(a+1,b,a+b+1)$ clearly have a connected orbit graph. As this holds for any $a\in\{1,\ldots,s-1\}$, the result is true of any two or more consecutive pairs of difference triples.
\end{proof}

\begin{lemma}\label{Lemma:1FactorMainComponent}
Suppose a graph $G$ has $1$-factorisation $\mathcal{F}=\{F_1,\ldots,F_n\}$ such that $F_i\cup F_j$ is a Hamilton cycle on $G$ for some $i,j\in\{2,\ldots,n\}$, and let $U=\{\infty_1,\ldots,\infty_n\}$ be a set where $V(G)\cap U=\emptyset$. 
Let $\mathcal{B}=\left\{F_i\vee\infty_i:i\in \left\{1,\ldots,n \right\}\right\}\cup \left\{F_{i+1}\vee\infty_{i}:i \in \left\{1,\ldots,n-1 \right\}\right\}$. 
Then the $2$-BIG of the partial TTS given by $(V(G)\cup U, \mathcal{B})$ is connected and bipartite.
\end{lemma}

\begin{proof}
The $2$-BIG of $(V(G)\cup U, \mathcal{B})$ is bipartite by Lemma~\ref{Lemma:BipartiteTTS}, so it remains to show that the $2$-BIG is connected.
 Let $w=|V(G)|$, and without loss of generality we assume $F_n\cup F_{n-1}$ is a Hamilton cycle on $G$. 
Then the blocks given by $F_{n-1}\vee\infty_{n-1}$ and $F_{n}\vee\infty_{n-1}$ form a $w$-cycle in the $2$-BIG. Moreover, each block in $F_i\vee\infty_i$ has an edge to at least one block in each of $F_i\vee\infty_{i-1}$ and $F_{i+1}\vee\infty_i$.
\end{proof}

\begin{lemma}\label{Lemma:PathWithinOrbit}
Let $w$ be an even integer. Let $\mathcal{D}$ be a set of difference triples that generates the simple partial TTS($w$) ($\Z_w,\mathcal{T}$). If there exist distinct $a,b,c\in \{1,2,\ldots,\frac{w}{2}-1\}$ such that $\{(a,b,a+b),(a,c,a+c),(a+1,c-1,a+c),(a+1,b,a+b+1)\}\subseteq\mathcal{D}$, then for all $x\in\Z_w$ there is a path from $x+\{0,b,a+b\}$ to $(x+1)+\{0,b,a+b\}$ in the $2$-BIG of $\mathcal{T}$.
\end{lemma}

\begin{proof}
Taking $x=0$, the $2$-BIG of $\mathcal{T}$ contains the path
\[  [\{0,b,a+b\},\{0,b,a+b+1\},\{b-c+1,b,a+b+1\},\{b-c+1,b+1,a+b+1\},\{1,b+1,a+b+1\}].\]
Likewise, because blocks of $\mathcal{T}$ are cyclically generated from difference triples in $\mathcal{D}$, then for any $x\in\Z_w$ it is true that the $2$-BIG contains a path from $x+\{0,b,a+b\}$ to $x+\{1,b+1,a+b+1\}$.
\end{proof}

\begin{lemma}\label{Lemma:ConstructionConnected} 
Suppose $u$ and $v$ are admissible integers such that $v>2u$ and $u>13$.
The $2$-BIG corresponding to the partial TTS($v$), $(V,\mathcal{R}_1\cup\mathcal{R}_2)$ given by Construction~\ref{Construction:PartialTTS} is connected and bipartite. Furthermore, $(V,\mathcal{R}_1\cup\mathcal{R}_2)$ can be completed to a TTS($v$) by adding the blocks of a TTS($u$) with vertex set $U\subset V$, and there exist distinct points $\a,\b,\g\in V\setminus U$ such that each block in $\mathcal{R}_1\cup\mathcal{R}_2$ contains at most one of $\a$, $\b$ and $\g$.
\end{lemma}

\begin{proof}Let $\mathcal{T}=\mathcal{R}_1\cup\mathcal{R}_2$. Let $w=v-u$ and let $k$, $t$ and $s$ be integers such that $w=12t+k$ where $k\in\{0,2,\ldots,10\}$, $s=2t$ if $k\geq 6$ and $s=2t-1$ if $k\leq 4$.

The partial TTS($v$) ($V,\mathcal{T}$) given by Construction~\ref{Construction:PartialTTS} is the union of two partial STS($v$) with block sets $\mathcal{R}_1$ and $\mathcal{R}_2$ so the $2$-BIG is bipartite. Since we remove at least one difference triple  in Step 2 of Construction~\ref{Construction:PartialTTS}, it follows that there are three points $\a,\b,\g\in V\setminus U$ such that each block in $\mathcal{R}_1\cup\mathcal{R}_2$ contains at most one of $\a$, $\b$ and $\g$. We now show that ($V,\mathcal{T}$) has a connected $2$-BIG.

We first observe that by Lemma~\ref{Lemma:1FactorMainComponent} the blocks formed from the $1$-factorisation of $\mathcal{L}'$ are all in the same component. 
If $w-u\in\{1,3,5\}$ then these are the only blocks in the partial TTS($v$), and it follows that the $2$-BIG is connected. Suppose $w-u\geq 7$.

{\bf Case 1.} $w-u\in\{7,9,11\}$. 
The blocks in $F_i\vee \infty_j$  for $4\leq j\leq i\leq u$ and $G_1\vee\infty_u$ are all in the same component of the $2$-BIG by Lemma~\ref{Lemma:1FactorMainComponent}. Furthermore each block $\{a,b,c\}\in\mathcal{R}_1\cup\mathcal{R}_2$ is adjacent to a block of the form $\{a,b,\infty_i\}$, and for each $x\in\{0,2,\ldots,w-2\}$ the $2$-BIG contains the following paths 
\[\begin{array}{l}
x+[\{0,c_1,\infty_1\},\{0,c_1,d_1\},\{0,d_1,\infty_3\}],\\
x+[\{1,1+c_1,\infty_2\},\{1,1+c_1,1+d_1\},\{1,1+d_1,\infty_4\}],\\
x+[\{0,c_2,\infty_u\},\{0,c_2,d_2\},\{0,d_2,\infty_3\}].
\end{array}\]

{\bf Case 2.} $w>u+11$. 
Since $\mathcal{D}_1$ and $\mathcal{D}_2$ are sets of consecutive difference triples given by Constructions~\ref{Construction:sOddSL}--\ref{Construction:sEvenCP}, it follows from Proposition~\ref{Prop:ConnectedOrbitGraph} that the corresponding orbit graph is connected. 

Finally, we show that there exists a difference triple $D\in \mathcal{D}_1\cup\mathcal{D}_2$  such that the  $2$-BIG contains a path from $x+B$ to $(x+1)+B$ for $x\in\{0,\ldots,w-1\}$, where $D$ generates blocks $B,1+B,\ldots,(w-1)+B$. Since $h\leq s-2$, at least two difference triples that satisfy the conditions of Lemma~\ref{Lemma:PathWithinOrbit} remain in each of Constructions~\ref{Construction:sOddSL}--\ref{Construction:sEvenCP}. 
In particular, for the case where $w\equiv 2,\, 4\mod{12}$ and $t$ is even let $a=2$, $b=5t-1$ and $c=3t-1$. When $t$ is odd and $w\equiv 4\mod{6}$ then let $a=2t-2$, $b=4t+1$ and $c=2t+1$. When $w\equiv 0\mod{12}$ or when $w\equiv 2\mod{6}$ and $t$ is odd then let $a=2t-3$, $b=2t+1$ and $c=4t+2$. If $w\equiv 6\mod{12}$ then $a=2t-2$, $b=2t+2$ and $c=4t+4$. When $w\equiv 8,\, 10\mod{12}$, let $a=2t-1$, $b=4t+3$ and $c=2t+2$.
Thus, by Lemma~\ref{Lemma:PathWithinOrbit} there is some orbit $B,1+B,\ldots,(w-1)+B$ such that the $2$-BIG contains a  path from $x+B$ to $(x+1)+B$ for $x\in\{0,\ldots,w-1\}$.
\end{proof}


Lemma~\ref{Lemma:ConstructionConnected} implies the following result for embedding TTS($u$). 

\begin{theorem}\label{Theorem:EmbeddingTTS}
If $u>13$ and there exists a TTS($u$), $(U,\mathcal{T})$ whose $2$-BIG is a bipartite connected graph, then for all admissible orders $v>2u$, there exists a TTS($v$),  $(V,\mathcal{T}')$ where $U\subset V$ and $\mathcal{T}\subset \mathcal{T}'$, and whose $2$-BIG is bipartite and contains exactly two components. Furthermore, there exist three points $a,b,c\in V\setminus U$ such that each block of $\mathcal{T}'$ contains at most one of $a$, $b$ and $c$.
\end{theorem}

\begin{proof}[Proof of Theorem~\ref{Theorem:ExtendingnonHamiltonianTTS}]
Let $(U,\mathcal{T})$ be the TTS($u$), where $U=\{\infty_1,\ldots,\infty_u\}$ and let $\mathcal{G}$ be the $2$-BIG of $(U,\mathcal{T})$. Let $w=v-u$ and let $(V,\mathcal{T}')$ be the TTS($v$) obtained by applying Theorem~\ref{Theorem:EmbeddingTTS} to $(U,\mathcal{T})$ where $V=U\cup\mathbb{Z}_{w}$.

We then perform the following trade to obtain a TTS($v$) whose $2$-BIG is connected. It follows from Theorem~\ref{Theorem:EmbeddingTTS} that there exist elements $a,b$ and $c$ in $\Z_w$ and distinct $\infty_i,\infty_j$ and $\infty_k$ in $U$ so that $\{a,b,c\}$ is not a block in $\mathcal{T}'$, but that one of the STS($v$) does contain the following blocks ({\it i.e.} the TTS($v$) contains these blocks and they are all in the same part of the bipartite partition): \[\{a,b,\infty_i\},\{b,c,\infty_j\},\{c,a,\infty_k\},\{\infty_i,\infty_j,\infty_k\}.\]

We remove these blocks from the TTS($v$) and replace them with
\[\{\infty_i,\infty_j,b\},\{\infty_i,\infty_k,a\},\{\infty_j,\infty_k,c\},\{a,b,c\}.\]

Since we performed the trade on blocks in one of the STS($v$), the resulting $2$-BIG is still bipartite. Furthermore, it is connected since, for example $\{\infty_i,\infty_j,b\}$ is adjacent to blocks in each of the two original components.

Suppose toward a contradiction that the $2$-BIG of the resulting TTS($v$) contains a Hamilton cycle. Then by reversing the trade given above, $\mathcal{G}$ must contain a Hamilton cycle which is a contradiction.\end{proof}

\section{Open Problems}

We conclude by discussing some open problems that arise naturally from this paper. Lemma~\ref{Lemma:smallv} states that there is no TTS($v$) with a connected bipartite non-Hamiltonian $2$-BIG for all admissible $v\leq 13$. This was proved by exhaustive computer search, a task which becomes impractical for larger orders. We examined several instances of decomposable  TTS($15$) and while each example had a 2-BIG that was either Hamiltonian or disconnected, many cases remain unknown. This leads us to the following question. 

\begin{question}
What is the smallest admissible integer $v$ such that there exists a TTS($v$) with a connected bipartite non-Hamiltonian $2$-BIG? 
\end{question}

It is possible that the construction of the TTS(331) in Section~\ref{Section:TTS331} could be further optimised. For example, at each splicing step we introduce new points to guarantee that the resulting partial TTS is simple, however there may exist a relabelling of the configuration that uses fewer new points. Furthermore, the application of Lindner's construction (see Theorem~\ref{Theorem:Lindner}) does not necessarily provide a minimal embedding of the partial TTS(55) given by configuration $\mathbb{F}$. As mentioned when applying this construction, it could also be possible to embed $\mathbb{F}$ in a TTS($v$) for admissible $v>331$. However there appears to be no effective way to incorporate into this construction the requirement that the final 2-BIG is connected. These observations suggest that the upper bound in Theorem~\ref{Theorem:ExistenceLargev} can be improved from 663. 
Another open problem is as follows.

\begin{question}
What is the smallest  integer $N$ such that for all admissible $v\geq N$ there exists a TTS($v$) with a connected bipartite non-Hamiltonian $2$-BIG? 
\end{question}

For $v<12$ it is known whether a TTS($v$) has a Hamiltonian $2$-BIG. In particular, for $v\in\{3,6\}$ the $2$-BIG of a TTS($v$) is non-Hamiltonian and for $v\in\{4,7,9,10\}$ the $2$-BIG of a TTS($v$) is Hamiltonian. We also note that when $v=13$, if the $2$-BIG is bipartite and connected, this is a sufficient condition to guarantee that the $2$-BIG is Hamiltonian. For each order $v\geq 12$ such that $v\equiv 0$ or $1\mod{3}$, there exists a TTS($v$) whose 2-BIG is Hamiltonian and another whose 2-BIG is non-Hamiltonian. We therefore conclude with the following, somewhat ambitious, open problem.

\begin{question}
For $v\geq 12$ such that $v\equiv 0$ or $1\mod{3}$, find sufficient conditions for a TTS($v$) to have a Hamiltonian $2$-BIG.  
\end{question}

{\Large{ \bf Acknowledgements}}

The first author was supported by an AARMS postdoctoral fellowship.
The second author acknowledges research grant support from NSERC (grant number RGPIN-04456-2016), CFI and IRIF, as well as computational support from
Compute Canada and its consortia (especially ACENET, SHARCNET and
WestGrid) and The Centre for Health Informatics and Analytics of the Faculty
of Medicine at Memorial University of Newfoundland.

\end{document}